\documentclass[letterpaper, 10 pt, conference]{ieeeconf}  % Comment this line out if you need a4paper
\usepackage[noadjust]{cite}
\newcommand*{\QEDB}{\hfill\ensuremath{\square}}%  empty square

\makeatletter
\let\NAT@parse\undefined
\makeatother

\usepackage[utf8]{inputenc}
\usepackage[T1]{fontenc}
\usepackage{lmodern}
\usepackage[english]{babel}
\usepackage{amsmath, amssymb}
\usepackage{graphicx}
\let\labelindent\relax
\usepackage{enumitem}
\usepackage{bm}
\usepackage{dsfont}
\usepackage{wrapfig}
\usepackage[dvipsnames]{xcolor}
\usepackage[hidelinks]{hyperref}
\usepackage{booktabs}
\usepackage{multirow}

\newif\ifarxiv
\arxivtrue

\definecolor{myblue}{HTML}{006AA3}
\definecolor{mygreen}{HTML}{5B9A64}
\definecolor{myred}{HTML}{C1121F}
\definecolor{myorange}{HTML}{FAA00F}
\definecolor{mypurple}{HTML}{6D4C94}

\newcommand{\R}{\mathbb{R}}

\newtheorem{theorem}{Theorem}

\newtheorem{proposition}{Proposition}

\newtheorem{lemma}{Lemma}

\newtheorem{example}{Example}

\let\oldfrac\frac % Save the original frac command just in case
\renewcommand{\frac}[2]{%
  \mathchoice
    {\oldfrac{#1}{#2}}%
    {{#1}/{#2}}%
    {{#1}/{#2}}%
    {{#1}/{#2}}%
}
\usepackage{tikz}
\usepackage{pgfplots}
\pgfplotsset{compat=1.18}
\usepgflibrary{arrows.meta}

\usetikzlibrary{
    arrows.meta, positioning, backgrounds, shapes, angles,
    fit, hobby, decorations.markings, 3d, calc, quotes, shadows
}
\pgfdeclarelayer{background}
\pgfsetlayers{background,main}
\usepgfplotslibrary{fillbetween}

\tikzset{
    node style/.style={circle, draw, fill=white, inner sep=1pt, minimum size=1.5em},
    main line/.style={thick}
}
\usepackage{subcaption}
\usepackage{nicematrix}

\IEEEoverridecommandlockouts                              % This command is only needed if 
\title{\LARGE \bf
Worst-Case Finite-Horizon Analysis of Perturbed LQR via the Performance Estimation Problem
}

\author{Anthony Couthures, Gianluca Bianchin, and Julien M. Hendrickx% <-this % stops a space
\thanks{
The authors are with the ICTEAM Institute, UCLouvain, Belgium. G. Bianchin is also an Investigator of the WEL Research Institute, Belgium.
This research is supported in part by the “SIDDARTA” Concerted Research Action (ARC) granted by UCLouvain and in part by the “InterpoControl” F.R.S.-FNRS research project (PDR). Generative AI tools (Gemini 3.1, Claude Opus 5.5) were used to improve the code used in the numerical investigations. Corresponding author: {\tt\small anthony.couthures@uclouvain.be}}%
}

\begin{document}
\bstctlcite{IEEEexample:BSTcontrol}

\maketitle
\thispagestyle{empty}
\pagestyle{empty}

%%%%%%%%%%%%%%%%%%%%%%%%%%%%%%%%%%%%%%%%%%%%%%%%%%%%%%%%%%%%%%%%%%%%%%%%%%%%%%%%

\begin{abstract}
    We propose a framework to assess the robustness of a given state-feedback controller by evaluating its worst-case finite-horizon cost over a class of norm-bounded perturbations. To compute this cost, we show how to apply the Performance Estimation Problem (PEP) framework, originally developed to analyze optimization algorithms, and reformulate the worst-case search as a rank-constrained Semidefinite Program (SDP). Our results suggest that the perturbations that induce the highest finite-horizon costs can differ substantially from those identified using standard asymptotic metrics. Intuitively, worst-case perturbations are obtained by aligning slowly decaying modes with specific initial conditions to drive transient growth, causing a bifurcation in the worst-case behavior as the horizon length changes. These findings open new possibilities to certify finite-horizon performance of broader system classes, including data-driven models or nonlinear components.
\end{abstract}

%%%%%%%%%%%%%%%%%%%%%%%%%%%%%%%%%%%%%%%%%%%%%%%%%%%%%%%%%%%%%%%%%%%%%%%%%%%%%%%%
\section{Introduction}
\label{sec:intro}

Many control applications must guarantee performance over a finite time window, such as a maneuver, a batch process, or a mission of fixed duration. Over such horizons, the relevant measure is the accumulated transient cost rather than asymptotic stability. Feedback laws, such as a Linear Quadratic Regulator (LQR), are typically designed on a nominal model, whereas the physical system is subject to parametric uncertainties. Such parameter variations, which create a mismatch between the design model and the true physical system, can arise, for example, from changing operating conditions or unmodeled nonlinearities like a varying resistance in an electronic circuit.

To handle these parametric uncertainties, it is natural to turn to robust control to provide guarantees on the desired behaviors \cite{amatoFinitetimeControlLinear2001,mayneRobustModelPredictive2005,abelPrescribedTimeSafetyDesign2024}. Robust control approaches include, for example, $\mu$-analysis \cite{packardComplexStructuredSingular1993} (which evaluates worst-case stability and performance against structured uncertainties), robust Lyapunov functions and Integral Quadratic Constraints (IQC) \cite{megretskiSystemAnalysisIntegral1997} (a framework that bounds the input-output behavior of complex uncertainties and nonlinearities using quadratic inequalities). These methods provide theoretical certificates in the presence of such perturbations.

However, while powerful, most robust control certificates evaluate only asymptotic behavior or specific perturbations, limiting their use in finite-horizon applications. For instance, $\mu$-analysis computes exact robust stability margins for Linear Time-Invariant (LTI) systems \cite{packardComplexStructuredSingular1993}, but it evaluates asymptotic stability (see Example~\ref{ex:no_pole_motion}). Similarly, while recent structured pseudospectra algorithms provide robust transient bounds \cite{guglielmiRank1MatrixDifferential2023,guglielmiTransientDynamicsStructured2025}, they rely on continuous-time resolvent envelopes rather than finite-horizon trajectories. Time-domain IQC methods handle finite horizons using dissipation inequalities \cite{seilerFiniteHorizonRobustness2019}, but they rely on dual bounding techniques, which present two key limitations. First, standard time-domain IQC implementations often use static (memoryless) multipliers \cite{megretskiSystemAnalysisIntegral1997}, which mathematically relax the static LTI parameter variations into a larger class of Linear Time-Varying (LTV) uncertainties \cite{seilerFiniteHorizonRobustness2019}, introducing conservatism. Dynamic multipliers (e.g., Finite Impulse Response filters) can reduce this conservatism \cite{veenmanIQCSynthesisGeneralDynamic2011,fetzerInvarianceDynamicMultipliers2018}, but their basis, structure, and order must be chosen by the user, which is nontrivial for finite-horizon problems. Second, the dual approach acts primarily as a bounding oracle. Because the optimization variables are dual multipliers rather than system trajectories, extracting the exact physical worst-case perturbation matrices or initial conditions from the dual solution remains difficult.

By contrast, in this paper we evaluate the finite-horizon performance of a fixed state-feedback controller under norm-bounded LTI uncertainties using the Performance Estimation Problem (PEP) framework \cite{droriPerformanceFirstorderMethods2014, taylorSmoothStronglyConvex2017} and exhibit worst-case behaviors. Originally developed to analyze optimization algorithms, PEP treats worst-case analysis as an optimization problem over discrete trajectories. Using necessary and sufficient operator interpolation conditions, we formulate the search for worst-case LTI trajectories as a Quadratically Constrained Quadratic Program (QCQP) and then as a rank-constrained Semidefinite Program (SDP). This primal approach avoids LTV conservatism. By retaining the sequence variables, we can extract the exact worst-case initial conditions and perturbation matrices. Instead of just providing scalar bounds, PEP acts as a diagnostic tool to uncover geometric failure modes when these controllers are deployed over a finite horizon. Finally, we compare our relaxed SDP and numerically extracted lower bounds with the IQC bound using static multipliers, which quantifies the conservatism introduced by the LTV relaxation, and we relate the relaxed SDP to IQC bounds with richer multipliers through Lagrangian duality.

\textbf{Notation.} We denote by $\R$ the set of real numbers. Vectors of dimension $n$ will be denoted as $x \in \R^n$ in lowercase and a matrix $A \in \R^{n \times m}$ in uppercase letters. The concatenation of a sequence of vectors $\{x_k\}_{k=0}^{N-1}$ is denoted by $\bm{x} = [x_0, \dots, x_{N-1}] \in \R^{n \times N}$. For two matrices $A,B \in \R^{n \times n}$, we denote by $A \succeq B$ the Loewner partial order, i.e., $A-B$ is positive semidefinite.

\section{Problem Formulation}
\label{sec:problem_formulation}

Consider a discrete-time linear system $x_{k+1} = Ax_k + B u_k $ with state trajectory $\{x_k\}_{k=0}^{N} \subset \R^n$, control sequence $\{u_k\}_{k=0}^{N-1} \subset \R^m$, for the nominal matrices $A \in \R^{n \times n}$ and $B \in \R^{n \times m}$. Let $u_k = K x_k$ be a nominal state-feedback control law, with $K$ any fixed gain (in our examples, the infinite-horizon LQR gain).

As discussed in Section~\ref{sec:intro}, in many control applications, real-world systems are characterized by system matrices that differ from the nominal matrices used to design the controller. We therefore evaluate the performance of the fixed controller $K$ under unknown but bounded LTI perturbations on the system matrices:
\begin{equation}\label{eq:perturbed_plant}
    x_{k+1} = (A + \Delta_A)x_k + (B + \Delta_B) u_k
\end{equation}
where $\Delta_A \in \R^{n \times n}$ and $\Delta_B \in \R^{n \times m}$ are static perturbations bounded in spectral norm by $\varepsilon_A, \varepsilon_B \ge 0$:
\begin{align*}
    \Delta_A & \in \mathcal{D}_A := \{ M \in \R^{n \times n} \mid \|M\|_2 \le \varepsilon_A \}, \\
    \Delta_B & \in \mathcal{D}_B := \{ M \in \R^{n \times m} \mid \|M\|_2 \le \varepsilon_B \}.
\end{align*}
Under the nominal control law $u_k = Kx_k$, the closed-loop perturbed dynamics becomes:
\begin{equation}\label{eq:closed_loop_dynamics}
    x_{k+1} = (\Phi + \Delta_A + \Delta_B K) x_k,
\end{equation}
where $\Phi = A + BK$ is the nominal closed-loop state matrix.

To quantify the transient performance degradation caused by these uncertainties, we consider a finite-horizon evaluation cost over $N$ steps, with stage cost penalties
$Q \succeq 0$,  $R \succ 0$, and terminal penalty $S \succeq 0$:
\begin{equation*}
    J_N(x_0, \Delta) = \sum_{k=0}^{N-1} \left( x_k^\top Q x_k + u_k^\top R u_k \right) + x_N^\top S x_N,
\end{equation*}
where the state trajectory depends on the perturbations $\Delta = (\Delta_A, \Delta_B)$ through the closed-loop dynamics defined in \eqref{eq:closed_loop_dynamics}. Substituting $u_k = K x_k$, this cost simplifies to a quadratic function of the state trajectory:
\begin{equation}
    J_N(x_0, \Delta) = \sum_{k=0}^{N-1} x_k^\top Q_K x_k + x_N^\top S x_N,
    \label{eq:finite_cost}
\end{equation}
where $Q_K = Q + K^\top R K \succeq 0$ is the equivalent closed-loop stage cost matrix, and the state sequence $\{x_k\}_{k=0}^N$ is generated by the perturbed dynamics \eqref{eq:closed_loop_dynamics}.

Our objective is to quantify the largest finite-horizon cost of the fixed controller $K$ and identify the initial states and admissible LTI perturbations that produce it. We are therefore interested in determining solutions to the following optimization problem\footnote{
    The unit-ball constraint normalizes the initial-state magnitude without loss of generality: by linearity of the dynamics and quadratic scaling of the cost, replacing it with $\|x_0\|_2^2 \le \gamma$ scales the worst-case cost to $\gamma J_N^*$ for any $\gamma \ge 0$.
}:
\begin{align}
                & J_N^*(\varepsilon_A, \varepsilon_B, Q_K, S) := \max_{\{x_k\}_{k=0}^N, \Delta_A, \Delta_B} \quad J_N(x_0, \Delta) \nonumber \\
    \text{s.t.} & \quad  \|x_0\|_2^2 \le 1     \label{eq:exact_worst_case}                                                                   \\
                & \quad x_{k+1} = (\Phi + \Delta_A + \Delta_B K) x_k, \quad \forall k \in \mathcal{K}  \nonumber                             \\
                & \quad \Delta_A \in \mathcal{D}_A, \quad \Delta_B \in \mathcal{D}_B \nonumber
\end{align}
where $\mathcal{K} =\{0, \dots, N-1\}$. One may expect the worst case to be produced by the perturbation that moves the closed-loop poles the most. The following example shows that this is not always the case.

\begin{example}\label{ex:no_pole_motion}
    \textit{(Pole locations do not capture transient performance)}
    In this example, we illustrate that worst-case perturbations can increase the finite-horizon cost \eqref{eq:exact_worst_case} without moving the closed-loop poles of the system with state feedback $K$ taken as the optimal gain of infinite horizon LQR problem, i.e., $K = -(R + B^\top P B)^{-1} B^\top P A$, where $P \succ 0$ solves the discrete algebraic Riccati equation $P = A^\top P A - A^\top P B (R + B^\top P B)^{-1} B^\top P A + Q$. Consider the following system:
    \begin{equation*}
        A=\begin{bmatrix}0 & 1.25 \\
                       0 & 0.25
        \end{bmatrix},\, \,
        B=\begin{bmatrix}
            0.5 \\
            0.5
        \end{bmatrix},\, \,
        Q=\mathrm{diag}(1,0.75),\, \, R=1.
    \end{equation*}
    Moreover, we let $N=2$, $S=0$, $\varepsilon_A = 0.375$ and $\varepsilon_B = 0.25$.
    The Riccati solution is $P=\mathrm{diag}(1,2)$, so that $K=[0 \; -0.5]$, $\Phi=\left[\begin{smallmatrix}0&1\\0&0\end{smallmatrix}\right]$ and $Q_K=\mathrm{Id}$.

    Since $N=2$ and $S=0$, only $x_0$ and $x_1$ are penalized, so the worst case maximizes the one-step amplification of the perturbed closed loop. Writing $\Phi_\Delta = \Phi + \Delta_A + \Delta_B K$, the cost reduces to $J_2(x_0,\Delta) = \|x_0\|_2^2 + \|\Phi_\Delta x_0\|_2^2$, hence $J_2^*=1+\max_{\Delta}\|\Phi_\Delta\|_2^2$.
    The triangle inequality gives $\|\Phi_\Delta\|_2 \le \|\Phi\|_2 + \varepsilon_A + \varepsilon_B \|K\|_2 = 1.5$, and this bound is attained by the following admissible perturbations
    \begin{equation*}
        \Delta_A^* = \begin{bmatrix}
            0 & 0.375 \\
            0 & 0
        \end{bmatrix},\, \, \Delta_B^* = \begin{bmatrix}
            -0.25 \\
            0
        \end{bmatrix},\, \, \Phi_{\Delta^*}=\begin{bmatrix}
            0 & 1.5 \\
            0 & 0
        \end{bmatrix}.
    \end{equation*}
    From $x_0=[0\;\,1]^\top$, this gives $J_2^*=1+2.25=3.25$, against $2$ for the nominal loop: the perturbation increases the cost by more than 60\%.

    Yet $\Phi$ and $\Phi_{\Delta^*}$ have the same eigenvalues $\{0,0\}$: the worst-case perturbation leaves the poles in place and strengthens the coupling from the second state to the first. An analysis based on eigenvalue location, such as $\mu$-analysis or pseudospectra \cite{packardComplexStructuredSingular1993}, sees no degradation here. In our computations, the perturbations that maximize the spectral radius ($\rho\approx0.79$) yield at most $J_2\approx3.05$.
    \QEDB\end{example}

\section{The Performance Estimation Problem (PEP) Reformulation}
\label{sec:pep_framework}
Our main result, Theorem~\ref{thm:gram_formulation}, uses the PEP framework to recast the exact worst-case optimization problem \eqref{eq:exact_worst_case} into a rank-constrained SDP. To derive this formulation, we first abstract the problem over discrete sequences, enforce the LTI behavior of the perturbations through operator interpolation, and lift the sequence variables into a positive semidefinite Gram matrix.

\subsection{The Abstract Optimization Problem}

The products $\Delta_A x_k$ and $\Delta_B Kx_k$ in \eqref{eq:closed_loop_dynamics} make the constraints bilinear in the decision variables, and therefore nonconvex. To overcome this complexity, we treat the perturbation outputs as independent sequence variables. Over the horizon $N$, we define:
\begin{align*}
    \bm{x}   & = [x_0, \dots, x_{N-1}] \in \R^{n \times N},       \\
    \bm{v}_A & = [v_{A,0}, \dots, v_{A,N-1}] \in \R^{n \times N}, \\
    \bm{v}_B & = [v_{B,0}, \dots, v_{B,N-1}] \in \R^{n \times N},
\end{align*}
where at each time step $k \in \mathcal{K}$, the pointwise perturbations are given by:
\begin{equation*}
    v_{A,k} = \Delta_A x_k, \qquad v_{B,k} = \Delta_B K x_k.
\end{equation*}

By substituting these vectors into the system dynamics, the state update becomes a linear equation independent of the explicit matrices $\Delta_A$ and $\Delta_B$:
\begin{equation}\label{eq:additive_dynamics}
    x_{k+1} = \Phi x_k + v_{A,k} + v_{B,k}.
\end{equation}
Instead of optimizing over the matrices $\Delta_A, \Delta_B$, our decision variables are now the sequences of vectors: the states $\{x_k\}_{k=0}^{N}$ and the perturbations $\{v_{A,k}\}_{k=0}^{N-1}$ and $\{v_{B,k}\}_{k=0}^{N-1}$.

With this change, the optimization problem \eqref{eq:exact_worst_case} can thus be rewritten as:
\begin{align}
                      & J_N^*(\varepsilon_A, \varepsilon_B, Q_K, S) = \max_{\bm{x}, \bm{v}_A, \bm{v}_B} \quad  J_N(x_0, \bm{v}_A,\bm{v}_B) \nonumber \\
    \text{s.t.} \quad & \|x_0\|_2^2 \le 1 \label{eq:abstract_primal}                                                                                 \\
                      & x_{k+1} = \Phi x_k + v_{A,k} + v_{B,k}, \quad \forall k \in \mathcal{K}   \nonumber                                          \\
                      & \exists \Delta_A \in \mathcal{D}_A \text{ s.t. } v_{A,k} = \Delta_A x_k, \quad \forall k \in \mathcal{K}  \nonumber          \\
                      & \exists \Delta_B \in \mathcal{D}_B \text{ s.t. } v_{B,k} = \Delta_B K x_k, \quad \forall k \in \mathcal{K}  \nonumber        \\[-1em]\nonumber
\end{align}
By shifting to this sequence-based formulation, the objective and the dynamics are now quadratic and linear, respectively.

\subsection{Operator Interpolation: Enforcing LTI Behavior}

The last two constraints in \eqref{eq:abstract_primal} are
abstract, in the sense that they cannot be directly encoded in
a solver. In this section, we use tools from the PEP framework to translate these constraints into algebraic inequalities.

Let $\mathcal{L}_{[0, L]}$ denote the set of matrices $M \in \R^{d_2 \times d_1}$ with an induced $2$-norm bounded by $L \ge 0$. We use operator interpolation to enforce LTI behavior on the perturbation variables. Let $\bm{p} = [p_0, \dots, p_{N-1}] \in \R^{d_1 \times N}$ and $\bm{q} = [q_0, \dots, q_{N-1}] \in \R^{d_2 \times N}$ be matrices containing vector sequences. We say the sequences are $\mathcal{L}_{[0, L]}$-interpolable if there exists $M \in \mathcal{L}_{[0, L]}$ such that $q_k = M p_k$ for all $k$. For $\mathcal{L}_{[0, L]}$ the interpolation conditions are given by algebraic conditions over the data matrices from \cite{bousselmiInterpolationConditionsLinear2024}.

\begin{lemma}[$\mathcal{L}_{[0,L]}$-Interpolation Condition \cite{bousselmiInterpolationConditionsLinear2024}]\label{lem:interp_condition}
    Let $\bm{p} \in \R^{d_1 \times N}$ and $\bm{q} \in \R^{d_2 \times N}$ be data matrices. The sequences are $\mathcal{L}_{[0, L]}$-interpolable if and only if the following LMI holds:
    \begin{equation*}
        L^2 \bm{p}^\top \bm{p} - \bm{q}^\top \bm{q} \succeq 0.
    \end{equation*}
\end{lemma}

Lemma~\ref{lem:interp_condition} allows us to map the operator existence condition in \eqref{eq:abstract_primal} to quadratic constraints on the data matrix. This allows us to transform the optimization problem into a QCQP, as shown next.

\begin{proposition}[QCQP formulation]\label{prop:qcqp_formulation}
    The worst-case finite-horizon performance $J^*_N$ defined in \eqref{eq:abstract_primal} equals the optimal value of the following QCQP over the finite vector sequences $\bm{x}, \bm{v}_A$, and $\bm{v}_B$:
    \begin{align}
                          & J^*_N(\varepsilon_A, \varepsilon_B, Q_K, S)= \max_{\bm{x}, \bm{v}_A, \bm{v}_B} \quad   J_N(x_0, \bm{v}_A,\bm{v}_B) \nonumber \\
        \text{s.t.} \quad & \|x_0\|_2^2 \le 1 \label{eq:pep_QCQP}                                                                                        \\
                          & x_{k+1} = \Phi x_k + v_{A,k} + v_{B,k}, \quad \forall k \in \mathcal{K} \nonumber                                            \\
                          & \varepsilon_A^2 \bm{x}^\top \bm{x} - \bm{v}_A^\top \bm{v}_A \succeq 0 \nonumber                                              \\
                          & \varepsilon_B^2 \bm{x}^\top K^\top K \bm{x} - \bm{v}_B^\top \bm{v}_B \succeq 0 \nonumber                                     \\[-1em]\nonumber
    \end{align}
\end{proposition}

\begin{proof}
    Using Lemma~\ref{lem:interp_condition}, we map the existence constraints $\exists \Delta_A \in \mathcal{D}_A$ in the optimization problem \eqref{eq:abstract_primal} to the LMI $\varepsilon_A^2 \bm{x}^\top \bm{x} - \bm{v}_A^\top \bm{v}_A \succeq 0$. Similarly, $\exists \Delta_B \in \mathcal{D}_B$ is replaced by $\varepsilon_B^2 \bm{x}^\top K^\top K \bm{x} - \bm{v}_B^\top \bm{v}_B \succeq 0$, yielding \eqref{eq:pep_QCQP}.
\end{proof}

Proposition~\ref{prop:qcqp_formulation} provides an equivalent formulation of \eqref{eq:abstract_primal} and thus of \eqref{eq:exact_worst_case}. Notice that, in \eqref{eq:pep_QCQP}, the perturbation matrices no longer appear explicitly.

\subsection{The Gram Matrix SDP Formulation}\label{subsec:gram_sdp}

Unfortunately, the QCQP in \eqref{eq:pep_QCQP} remains a
non-convex optimization problem due to the quadratic matrix
inequalities.
To bypass this complexity, in this section, we use a Gram lifting approach \cite{luoSemidefiniteRelaxationQuadratic2010}. Because the objective function and constraints in \eqref{eq:pep_QCQP} depend solely on quadratic inner products of the sequence variables, we can linearize them by shifting our decision variables to a positive semidefinite Gram matrix.

We concatenate the initial state and the independent perturbation sequences into a single lifted vector $\bm{z}$:
\begin{equation*}
    \bm{z} = \begin{bmatrix} x_0^\top & v_{A,0}^\top \dots  v_{A,N-1}^\top & v_{B,0}^\top  \dots  v_{B,N-1}^\top \end{bmatrix}^\top.
\end{equation*}
We also define block-selector matrices $E_0, E_{A,k}, E_{B,k} \in \R^{n \times n(1+2N)}$ to extract the corresponding blocks from $\bm{z}$:
\begin{equation*}
    x_0 = E_0 \bm{z}, \quad v_{A,k} = E_{A,k} \bm{z}, \quad v_{B,k} = E_{B,k} \bm{z}.
\end{equation*}

Because the linear dynamics determines the state trajectory, every state $x_k$ is a linear combination of the components of $\bm{z}$. Using the dynamics from Proposition~\ref{prop:qcqp_formulation}, we define the state extraction matrices $H_{x,k}$ recursively:
\begin{equation*}
    H_{x,0} = E_0, \quad  H_{x,k+1} = \Phi H_{x,k} + E_{A,k} + E_{B,k}, \quad \forall k \in \mathcal{K}.
\end{equation*}

Thus, $x_k = H_{x,k} \bm{z}$ for all $k\in \mathcal{K}$. Any scalar inner product in the system can be expressed as a linear trace operation on the $(1+2N)n \times (1+2N)n$ positive semidefinite Gram matrix $Z = \bm{z} \bm{z}^\top$, which has rank at most 1. Indeed, for any symmetric $M \in \R^{n\times n}$, one has:
\begin{equation*}
    x_i^\top M x_j = \bm{z}^\top H_{x,i}^\top M H_{x,j} \bm{z} = \operatorname{Tr} \left( H_{x,j}^\top M H_{x,i} Z \right).
\end{equation*}

By the trace property, the $(i,j)$-th entry of this matrix is $\operatorname{Tr}(H_{x,j}^\top H_{x,i} Z)$. Similarly, the cost $J_N$ can be expressed as a linear operator acting on $Z$. Explicitly, the symmetric cost matrix $M_{J_{N}} \in \R^{(1+2N)n \times (1+2N)n}$ is obtained from the state extraction matrices $H_{x,k}$ as:
\begin{equation*}
    M_{J_N} = \sum_{k=0}^{N-1} H_{x,k}^\top Q_K H_{x,k} + H_{x,N}^\top S H_{x,N},
\end{equation*}
which yields $J_N(x_0, \bm{v}_A, \bm{v}_B) = \operatorname{Tr}(M_{J_N} Z)$.

Let $\mathcal{G}_{\bm{x}}(Z) \in \mathbb{S}^{N \times N}$ denote the linear operator mapping $Z$ to the $N \times N$ sequence inner-product matrix $\bm{x}^\top \bm{x}$. We define analogous linear operators $\mathcal{G}_{\bm{v}_A}(Z), \mathcal{G}_{\bm{v}_B}(Z)$, and $\mathcal{G}_{K\bm{x}}(Z)$ for the other sequences.

With those mathematical tools, we are now ready to present our main result, which expresses the optimization problem \eqref{eq:exact_worst_case} with the lifted variables $Z$.

\begin{theorem}[Rank-constrained SDP formulation]\label{thm:gram_formulation}
    The worst-case finite-horizon performance $J^*_N$ defined in \eqref{eq:abstract_primal} equals the optimal value of the following rank-constrained SDP over the Gram matrix $Z$:
    \begin{align}
                          & J^*_N(\varepsilon_A, \varepsilon_B, Q_K, S) = \max_{Z \succeq 0} \quad \operatorname{Tr} (M_{J_N} Z  ) \nonumber                        \\
        \text{s.t.} \quad & \operatorname{Tr}\left( E_0^\top E_0 Z \right) \le 1 \label{eq:exact_pep_primal}                                                        \\
                          & \varepsilon_A^2 \mathcal{G}_{\bm{x}}(Z) - \mathcal{G}_{\bm{v}_A}(Z) \succeq 0 \nonumber                                                 \\
                          & \varepsilon_B^2 \mathcal{G}_{K\bm{x}}(Z) - \mathcal{G}_{\bm{v}_B}(Z) \succeq 0 \nonumber                                                \\ \nonumber
                          & \operatorname{rank}(Z) \le 1                                                                                                  \nonumber
    \end{align}
\end{theorem}

Theorem~\ref{thm:gram_formulation} reformulates \eqref{eq:exact_worst_case} as an SDP whose only nonconvex constraint is the rank constraint. Dropping this constraint yields a tractable convex relaxation and an upper bound on $J_N^*$. The rank of an optimal relaxed solution $Z_{\mathrm{SDP}}^*$ then determines whether a worst-case trajectory can be recovered directly: a rank-one solution certifies that the bound is tight and yields a physically realizable worst-case trajectory. A higher-rank solution may instead represent a superposition of trajectories and need not correspond to any single admissible LTI system \cite{luoSemidefiniteRelaxationQuadratic2010}; in this case, the bound may be conservative.

\begin{figure*}[t]
    \centering

    \begingroup
    \setlength{\tabcolsep}{8pt}
    \renewcommand{\arraystretch}{1.2}
    \begin{tabular}{@{} l l l l@{}}
        \rule[0.6ex]{1.5em}{1pt} Unit Circle                                &
        \textcolor{mygreen}{\Large $\bullet$} Nominal eigenvalue            &
        \textcolor{red!70!black}{\Large $\bm{\star}$} Worst-case eigenvalue &
        \textcolor{myorange}{\rule{1.5em}{1.2ex}} $\Lambda_\varepsilon(\Phi)$ Structured $\varepsilon$-pseudospectrum
    \end{tabular}
    \endgroup

    \begin{subfigure}[b]{0.32\textwidth}
        \centering
        \includegraphics[width=\textwidth]{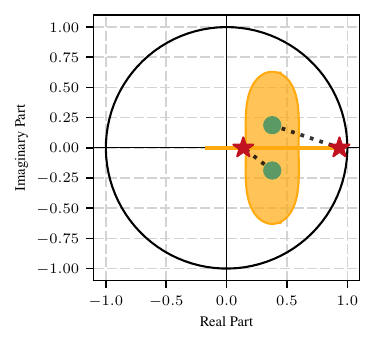}
        \caption{$A_1$, $\varepsilon_A =0.22$, $\varepsilon_B =0$: $\lambda^* > 0$}
        \label{fig:eig_mon}
    \end{subfigure}
    \hfill
    \begin{subfigure}[b]{0.32\textwidth}
        \centering
        \includegraphics[width=\textwidth]{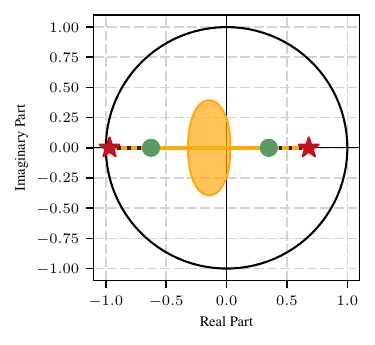}
        \caption{$A_2$, $\varepsilon_A =0.2$, $\varepsilon_B =0.1$: $\lambda^* < 0$}
        \label{fig:eig_osc}
    \end{subfigure}
    \hfill
    \begin{subfigure}[b]{0.32\textwidth}
        \centering
        \includegraphics[width=\textwidth]{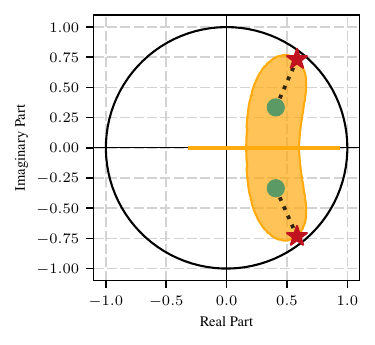}
        \caption{$A_3$, $\varepsilon_A =0.1$, $\varepsilon_B =0.25$: $\lambda^* \in \mathbb{C} \setminus \mathbb{R}$}
        \label{fig:eig_complex}
    \end{subfigure}

    \vspace{0.1em}
    \begingroup
    \setlength{\tabcolsep}{8pt}
    \renewcommand{\arraystretch}{1.2}
    \begin{tabular}{@{} c l c l @{}}

        % Item 1: Nominal Traj (Solid Green Line + Circle)
        \makebox[2.5em][c]{\textcolor{mygreen}{\rule[0.6ex]{2.5em}{1.2pt}}}% Draws solid line
        \hspace{-2.5em}% Backtracks to the start of the line
        \makebox[2.5em][c]{\large\textcolor{mygreen}{$\bullet$}} % Overlays circle in the center
         & Nominal Trajectory    &

        % Item 2: Worst-Case Traj (Dashed Red Line + Square)
        \makebox[2.5em][c]{\textcolor{red!75!black}{%
                \rule[0.6ex]{0.5em}{1.2pt}\hspace{0.2em}%
                \rule[0.6ex]{0.5em}{1.2pt}\hspace{0.2em}%
                \rule[0.6ex]{0.5em}{1.2pt}%
            }}% Draws dashed line (3 segments)
        \hspace{-2.5em}% Backtracks to the start of the line
        \makebox[2.5em][c]{\textcolor{red!75!black}{\footnotesize $\blacksquare$}} % Overlays square
         & Worst-Case Trajectory
    \end{tabular}
    \endgroup

    % --- ROW 2: Phase Portraits (State-Space / Time Domain) ---
    \begin{subfigure}[b]{0.32\textwidth}
        \centering
        \includegraphics[width=\linewidth]{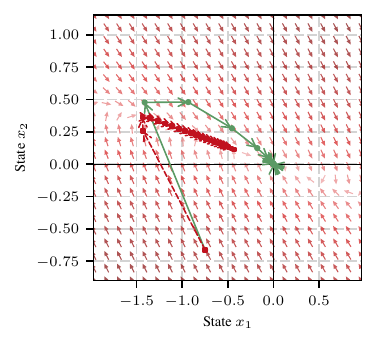}
        \caption{Monotonic 1D Slow Manifold}
        \label{fig:phase_mon}
    \end{subfigure}
    \hfill
    \begin{subfigure}[b]{0.32\textwidth}
        \centering
        \includegraphics[width=\linewidth]{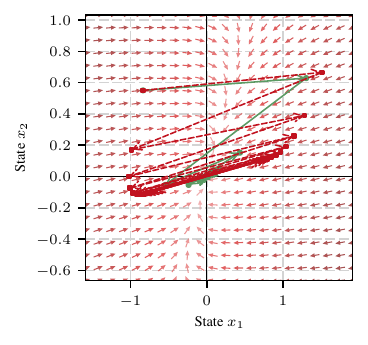}
        \caption{Oscillatory 1D Slow Manifold}
        \label{fig:phase_osc}
    \end{subfigure}
    \hfill
    \begin{subfigure}[b]{0.32\textwidth}
        \centering
        \includegraphics[width=\linewidth]{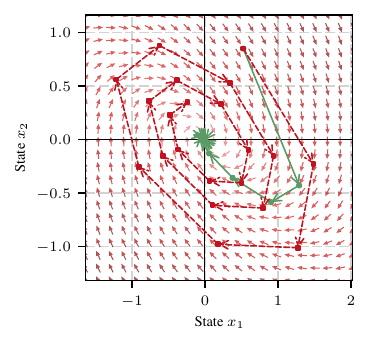}
        \caption{Spiraling 2D Invariant Subspace}
        \label{fig:phase_complex}
    \end{subfigure}

    \caption{Worst-case perturbation: spectral representation (top row) and trajectories (bottom row) extracted via PEP and projected gradient ascent.}
    \label{fig:pep_all_behaviors}
\end{figure*}

\section{Worst Case Perturbations via PEP}
\label{sec:numerical_discovery}

Section~\ref{subsec:gram_sdp} characterized the case where the relaxation is exact. In practice, however, the solution $Z_{\mathrm{SDP}}^*$ to the relaxed version of \eqref{eq:exact_pep_primal} often has rank strictly greater than one, even for relatively short planning horizons such as $N=5$, so that only the upper bound is available.
To obtain a lower bound together with an actual worst-case perturbation, we search directly over admissible perturbations. For fixed $(\Delta_A,\Delta_B)$, the trajectory is linear in $x_0$, so $J_N(x_0,\Delta) = x_0^\top G_N(\Delta)\, x_0$ with
\begin{equation*}
    G_N(\Delta) = \sum_{k=0}^{N-1} (\Phi_\Delta^k)^\top Q_K \Phi_\Delta^k + (\Phi_\Delta^N)^\top S\, \Phi_\Delta^N,
\end{equation*}
and $\Phi_\Delta$ as in Example~\ref{ex:no_pole_motion}. The worst $x_0$ in \eqref{eq:exact_worst_case} is therefore the dominant eigenvector of $G_N(\Delta)$, and the problem reduces to maximizing $\lambda_{\max}(G_N(\Delta))$ over $\mathcal{D}_A \times \mathcal{D}_B$. We do so by projected gradient ascent, following \cite{guglielmiRank1MatrixDifferential2023,guglielmiTransientDynamicsStructured2025}: each step moves $(\Delta_A,\Delta_B)$ along the gradient, then projects them back onto the norm balls by clipping their singular values at $\varepsilon_A$ and $\varepsilon_B$. As this problem is nonconvex, the ascent is started from several points: the perturbation that best fits the trajectory given by the dominant eigenvector of $Z_{\text{SDP}}^*$, and a set of random samples (Monte Carlo, MC) perturbations. In this case, every iterate is admissible, so the best cost found is a lower bound on $J_N^*$, while the SDP value is an upper bound. The ascent only guarantees first-order stationarity, but when the two bounds coincide, e.g., when $\operatorname{rank}(Z_{\text{SDP}}^*)=1$, the extracted perturbation and $x_0$ are the exact worst case.

We provide an interactive interface \cite{couthuresWorstcaseLTIPerturbations2026} for changing the parameters (matrices $A,B,Q,R,S$ and hyperparameters $N, \varepsilon_A,\varepsilon_B$) in real time using the solver MOSEK.

\subsection{Pseudospectral Targeting and Trajectories for large \texorpdfstring{$N$}{N}}\label{disc:pseudospectra}

To illustrate the PEP approach, we apply it to three 2D systems (Fig.~\ref{fig:pep_all_behaviors}). We chose the following nominal matrices $A$ because they each represent one of the three distinct worst-case behaviors we observed, making the differences visually clear:
\begin{equation*}
    A_1 = \begin{bmatrix} 1 & 1 \\ 0 & 1 \end{bmatrix},\, A_2 = \begin{bmatrix} -0.9 & 1 \\ 0 & 1 \end{bmatrix}\, \text{and}\, A_3 = \begin{bmatrix} 1 & 0.9 \\ -0.8 & 1 \end{bmatrix},
\end{equation*}
with $ B = \begin{bmatrix} 0 & 1 \end{bmatrix}^\top$, $Q = R = S = \mathrm{Id}$, $N=20$, and $K$ the infinite-horizon LQR gain computed with the same $Q$ and $R$.

Varying the nominal system matrices reveals three classes of worst-case behavior, which we interpret through their spectral geometry (Fig.~\ref{fig:pep_all_behaviors}, top row). This spectral view is used as a diagnostic of the PEP solution: in general, the finite-horizon objective is not determined by the perturbed eigenvalues alone (Example~\ref{ex:no_pole_motion}). For given uncertainty bounds $\varepsilon_A$ and $\varepsilon_B$, the structured $\varepsilon$-pseudospectrum of the closed-loop system \cite{guglielmiTransientDynamicsStructured2025} is the set of all possible eigenvalues for the perturbed matrix $\Phi_\Delta = \Phi + \Delta_A + \Delta_B K$ under valid LTI perturbations, computed exactly via the real stability radius formula \cite{qiuFormulaComputationReal1995}:
\begin{equation*}
    \begin{split} \Lambda_{\varepsilon}(\Phi) :=\{ \lambda \in \mathbb{C} \mid & \det(\lambda I - (\Phi + \Delta_A + \Delta_B K)) = 0,                                      \\
                                                                               & \text{for some } \|\Delta_A\|_2 \le \varepsilon_A, \, \|\Delta_B\|_2 \le \varepsilon_B \}.
    \end{split}
\end{equation*}
\begin{figure}[t]
    \centering
    \includegraphics[width=\linewidth]{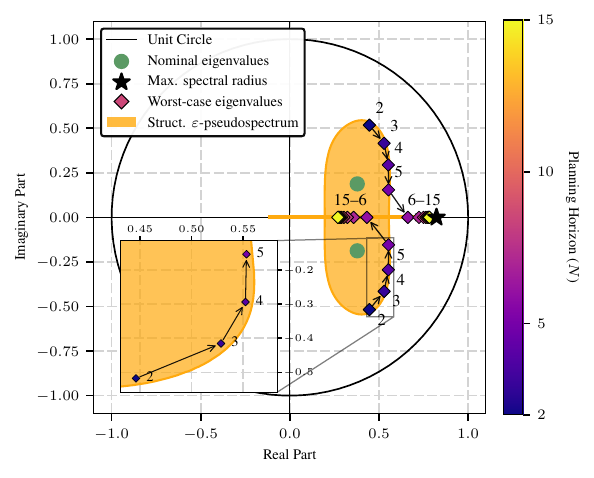}
    \caption{Spectral migration of the worst-case closed-loop poles as a function of the finite planning horizon $N$. The colored diamonds represent the worst-case poles extracted from the PEP solution.}
    \label{fig:pole_migration}
\end{figure}

As the red stars show, the optimal perturbation often pushes a dominant eigenvalue toward the boundary of $\Lambda_{\varepsilon}(\Phi)$. For large horizons, this behavior approaches the spectral-radius-maximizing perturbations predicted by classical $\mu$-analysis and pseudospectral theory \cite{hinrichsenRealComplexStability1990,packardComplexStructuredSingular1993}. However, these asymptotic tools do not fully characterize the finite-horizon optimum: PEP captures transient dynamics and invariant-direction effects, so the worst-case cost is not determined by the eigenvalues alone. In our numerical evaluations, the pseudospectral heuristic that maximizes the spectral radius of $\Phi_\Delta$ over the admissible perturbations \cite{guglielmiRank1MatrixDifferential2023} therefore yields a lower cost than PEP (e.g., on system $A_1$, $J^\rho_{20} \approx 20.95$ versus $J^{\mathrm{PEP}}_{20} \approx 21.06$). Since this maximizer is not unique, $J^\rho_N$ denotes the largest finite-horizon cost among the sampled maximizers, the most favorable choice for the heuristic.

This occurs because the finite-horizon cost depends on transient amplification, whereas asymptotic tools mainly focus on pole locations. PEP does not impose or target any particular spectral configuration: it optimizes directly over admissible finite-horizon trajectories. A posteriori, we observe that the resulting perturbations may remain slightly inside $\Lambda_{\varepsilon}(\Phi)$ depending on the cost matrices while modifying the associated invariant directions in a way that produces a larger transient cost.

The maximum spectral radius on the boundary can occur at a real positive, real negative, or complex eigenvalue, depending on the nominal eigenvalues. This shift changes the state-space trajectories (bottom row of Fig.~\ref{fig:pep_all_behaviors}). The perturbation modifies the dominant eigenspace of $\Phi_{\Delta^*}$ to confine the state energy to a slowly decaying invariant manifold. In our computations, we observe three cases:

\begin{itemize}[wide]
    \item \emph{Monotonic 1D Slow Manifold ($\lambda^* > 0$, Fig.~\ref{fig:phase_mon}):}
          When the dominant worst-case eigenvalue is real and positive, the optimal perturbation applies a rank-one shear opposing the nominal damping. The trajectory collapses onto a 1D invariant line (a ``slow manifold'') and moves monotonically toward the origin.

    \item \emph{Oscillatory 1D Slow Manifold ($\lambda^* < 0$, Fig.~\ref{fig:phase_osc}):}
          If the perturbed eigenvalue on the boundary is real and negative, the trajectory again collapses onto a 1D line. The negative eigenvalue causes the state to alternate signs, jumping back and forth across the origin at each step while slowly decaying.

    \item \emph{Spiraling 2D Invariant Subspace ($\lambda^* \in \mathbb{C} \setminus \mathbb{R}$, Fig.~\ref{fig:phase_complex}):}
          When the worst-case eigenvalues form a complex conjugate pair, $\Phi_{\Delta^*}$ acts on a 2D invariant plane, and the state trajectory spirals. The perturbation changes the rotational frequency and minimizes radial contraction, creating a slowly decaying spiral.
\end{itemize}

In classical matrix perturbation theory, asymptotic variations that maximize the spectral radius correspond to rank-one (for real eigenvalues) or rank-two (for complex conjugate pairs) perturbations \cite{trefethenSpectraPseudospectraBehavior2005}. The PEP solutions recover this geometry. The primal formulation collapses finite-horizon trajectories onto the slow manifolds dictated by worst-case $\mu$-analysis.

\subsection{Pole-Placement bifurcation with respect to \texorpdfstring{$N$}{N}}

The optimal adversarial LTI perturbation depends directly on the length of the evaluation horizon. We illustrate this using the system $A_1$ with the same $B$, $Q$ and $R$ but terminal cost $S = \mathrm{diag}(0, 1)$, which penalizes only the second state component. We set the uncertainty bounds $\varepsilon_A = 0.15$ and $\varepsilon_B = 0.05$. The different worst-case pole placements under those settings are shown in Fig.~\ref{fig:pole_migration}.

As the horizon $N$ increases, the initial transient growth contributes less to the total cumulative cost $J_N$. The slow asymptotic decay of the state begins to drive the cost instead. As $N$ approaches $15$, the perturbation stops manipulating transient eigenvectors. The poles migrate toward the rightmost edge of the structured $\varepsilon$-pseudospectrum to maximize the asymptotic time constant. Along the way, the poles undergo a geometric bifurcation. For horizons $N \leq 5$, the optimal perturbation maintains a complex conjugate pole pair. For $N \geq 6$, the poles meet on the real axis and split. This implies that the worst-case trajectories transition from 2D spiraling behavior (Fig.~\ref{fig:phase_complex}) to motion along a 1D slow manifold (Fig.~\ref{fig:phase_mon}).

Evaluating controllers over finite horizons therefore requires more than identifying the worst-case asymptotic poles on the pseudospectrum boundary.

\section{Comparison with IQC and Pseudospectral Approaches}
\label{sec:iqc_comparison}

We now compare the worst-case performance obtained from the different approaches discussed in this paper: the IQC bound with static multipliers \cite{megretskiSystemAnalysisIntegral1997}, the relaxed PEP SDP, the worst-case LTI perturbation extracted by projected gradient ascent, and the pseudospectral heuristic defined in Section~\ref{disc:pseudospectra}. We evaluate these methods on system $A_1$ to solve problem \eqref{eq:exact_worst_case}, using the same matrix $B$ as in the previous section. The cost matrices are set to $Q=R=S=\mathrm{Id}$, with perturbation bounds $\varepsilon_A = 0.15$ and $\varepsilon_B = 0.05$. Fig.~\ref{fig:iqc_comparison} shows how the accumulated cost changes with the horizon $N$ under each approach.

\begin{figure}[t]
    \centering
    \includegraphics[width=\linewidth]{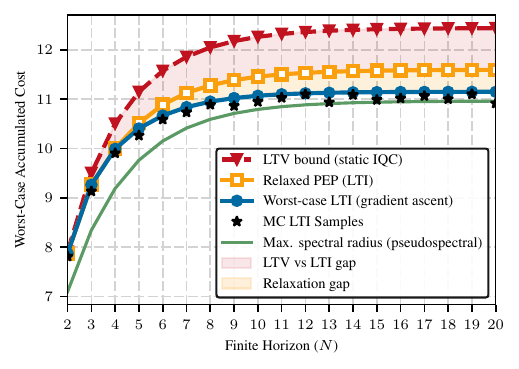}
    \caption{Finite-horizon accumulated cost bounds for system $A_1$ with $\varepsilon_A = 0.15$ and $\varepsilon_B = 0.05$. The LTV bound (static IQC) and the relaxed PEP (LTI) are upper bounds on $J_N^*$; the worst case found by gradient ascent, the Monte Carlo samples, and the pseudospectral heuristic are achievable costs, hence lower bounds.}
    \label{fig:iqc_comparison}
\end{figure}

To quantify the effect of time invariance, we also solve~\eqref{eq:exact_pep_primal} without the rank constraint and with each LMI replaced by its diagonal, i.e., $\varepsilon_A^2\|x_k\|_2^2 \geq \|v_{A,k}\|_2^2$ and $\varepsilon_B^2\|Kx_k\|_2^2 \geq \|v_{B,k}\|_2^2$ for all $k\in \mathcal{K}$. By Lemma~\ref{lem:interp_condition} applied at each time step separately, these are exactly the interpolation conditions of LTV perturbations $\Delta_{A,k}, \Delta_{B,k}$, so this bound (red) is an upper bound for the LTV problem, and the pink area shows the cost of this relaxation. By Lagrangian duality, it coincides with the IQC bound using static multipliers \cite{megretskiSystemAnalysisIntegral1997,seilerFiniteHorizonRobustness2019}. Similarly, the multipliers obtained by dualizing the full LMIs of~\eqref{eq:exact_pep_primal} weight these quadratic forms over all pairs of time steps, so the relaxed PEP (orange) coincides with the IQC bound using this larger class of multipliers, which contains the dynamic multipliers mentioned in Section~\ref{sec:intro}, as observed in \cite{scoyAbsoluteStabilityLifting2022}.

As discussed at the end of Section~\ref{subsec:gram_sdp}, dropping the rank constraint in the PEP SDP introduces conservatism, as it admits infeasible trajectories; this gap is shown by the orange shaded region. This gap initially grows as $N$ increases and then plateaus, similarly to the gap between the LTV and relaxed SDP bounds. However, the LTI worst case extracted by projected gradient ascent (blue circles) is admissible, so its cost is a lower bound on $J_N^*$. It is first-order stationary but not certified globally optimal. Its cost is higher than what random Monte Carlo sampling (black stars) typically finds, as random sampling struggles to hit the narrow geometric manifold where the worst-case parameters lie.

The pseudospectral heuristic (green solid line) yields the lowest cost of the group. As discussed in Section~\ref{disc:pseudospectra}, maximizing the spectral radius places the dominant eigenvalue on the boundary of the pseudospectrum and thus minimizes the long-term decay rate. However, it misses the short-term transient growth. The heuristic does not align the eigenvectors with a worst-case initial condition $x_0$, leaving out a large portion of the finite-horizon cost for short horizons.

\section{Conclusions and Perspectives}
\label{sec:conclusion}
This paper develops a Performance Estimation Problem (PEP) framework for evaluating the worst-case finite-horizon cost of a fixed state-feedback controller under norm-bounded LTI uncertainty. Operator interpolation is used to characterize the trajectories generated by a single admissible perturbation pair, yielding an exact rank-constrained semidefinite formulation. Its convex relaxation provides an upper bound on the worst-case cost. Our results suggest that finite-horizon worst cases are driven by transient alignment rather than pole placement: LTI perturbations align slowly decaying invariant manifolds with specific initial conditions. We also show that asymptotic metrics, such as structured pseudospectra, miss part of the finite-time cost and that IQC approaches can be conservative.
Open questions include characterizing when a worst-case LTV sequence collapses to an LTI one, leading to tight LTV relaxations.

\ifarxiv
    \section*{ACKNOWLEDGMENT}
    Anthony Couthures would like to thank \href{https://mengmouli.github.io/}{Mengmou Li} for insightful discussions and suggestions on the connections between the PEP framework and IQC analysis.
\fi
\bibliography{ECC}
\bibliographystyle{IEEEtran}

\addtolength{\textheight}{-12cm}                                     % on the last page of the document manually. It shortens

%%%%%%%%%%%%%%%%%%%%%%%%%%%%%%%%%%%%%%%%%%%%%%%%%%%%%%%%%%%%%%%%%%%%%%%%%%%%%%%%

%%%%%%%%%%%%%%%%%%%%%%%%%%%%%%%%%%%%%%%%%%%%%%%%%%%%%%%%%%%%%%%%%%%%%%%%%%%%%%%%


%%%%%%%%%%%%%%%%%%%%%%%%%%%%%%%%%%%%%%%%%%%%%%%%%%%%%%%%%%%%%%%%%%%%%%%%%%%%%%%%

\end{document}